\documentclass[11pt]{amsart}

 \newtheorem{thm}{Theorem}
 \newtheorem{cor}[thm]{Corollary}
 \newtheorem{lem}[thm]{Lemma}
 \newtheorem{prop}[thm]{Proposition}
 \theoremstyle{definition}
 \newtheorem{defn}[thm]{Definition}
 \theoremstyle{remark}
 
 \newcommand{\norm}[1]{\left\Vert#1\right\Vert}
 
 \newcommand{\C}{\mathbb{C}}

\begin{document}

\title[]
 {Reducing Subspaces on the Annulus}

\author{Ronald G. Douglas and Yun-Su Kim.  }

\address{Department of Mathematics, Texas AM University, College Station,
TX 77843-3368} \email{rdouglas@math.tamu.edu }

\address{Department of Mathematics, University of Toledo, Toledo,
OH 43606-3390} \email{Yun-Su.Kim@utoledo.edu}

\keywords{Bergman Spaces, Bilateral Weighted Shifts, Hardy Spaces,
Reducing Subspaces, MSC(2000) 47A15, 47B37, 47B38, 51D25.
\newline Research was partially supported by a grant from the National
Science Foundation.
 }

\dedicatory{}



\newpage

\begin{abstract}We study reducing subspaces for an analytic multiplication operator $M_{z^{n}}$
on the Bergman space $L_{a}^{2}(A_{r})$ of the annulus $A_{r}$,
and we prove that $M_{z^{n}}$ has exactly $2^n$ reducing
subspaces. Furthermore, in contrast to what happens for the disk,
the same is true for the Hardy space on the annulus. Finally, we
extend the results to certain bilateral weighted shifts, and
interpret the results in the context of complex geometry.

\end{abstract}

\maketitle





\section{Introduction}
Important themes in operator theory are determining invariant
subspaces and reducing subspaces for concretely defined operators.
Our goal in this note is to determine the reducing subspaces for a
power of certain multiplication operators on natural Hilbert
spaces of holomorphic functions on an annulus.

We begin with the Bergman space and Hardy space. Next, we consider
a generalization to certain bilateral weighted shifts. Finally, we
interpret our results in the context of complex geometry
describing another approach to these questions.

The motivation for these questions arises from some earlier
results of K. Zhu (\cite{Z}), M. Stessin and  K. Zhu (\cite{S}),
and other researchers (\cite{B}, \cite{N}). In these studies, the
annulus is replaced by the open unit disk, and one considers
$M_{z^n}$ on the Hardy space $H^2$ or the Bergman space
$L_{a}^{2}$. In particular, the lattice of reducing subspaces of
the $n$th power of the multiplication operator, $M_{z^n}$ on
$L^{2}_{a}$, was shown to be discrete and have precisely $2^n$
elements. This contrasted with the case of the classic Toeplitz
operator $T_{z^n}$ on the Hardy space $H^{2}$ for which this
lattice is infinite and isomorphic to the lattice of all subspaces
of ${\mathbb{C}}^n$. Thus, as is true for many other questions,
the situations on the
 unit disk and annulus are different.

For $0<r<1$, let $A_{r}$ denote the annulus $\{z\in\C:r<|z|<1\}$
in the complex plane $\C$. Let $L^{2}(A_{r})$ denote the usual
$L^2$-space for planar Lebesgue measure on $A_r$ and
$L_{a}^{2}(A_{r})$ be the closure of $R(A_{r})$ in $L^{2}(A_{r})$,
where $R(A_{r})$ is the space of all rational functions with poles
outside the closure of $A_r$.

We let $P_{L_{a}^{2}(A_{r})}$ be the orthogonal projection of
$L^{2}(A_{r})$ onto the Bergman space $L_{a}^{2}(A_{r})$.

For $\varphi$ in $H^{\infty}(A_{r})$, the space of bounded
holomorphic functions on $A_r$, define the operator $M_{\varphi}$
on $L_{a}^{2}(A_{r})$ so that
\[M_{\varphi}(f)=\varphi{f}\]
for $f$ in $L_{a}^{2}(A_{r})$. We are concerned with determining
the reducing subspaces of $M_{z^{n}}=M_{z}^{n}$ for $n\geq{2}$.

\section{Reducing Subspaces for
$M_{z^{n}}(n\geq{2})$}\label{28}

We let $\mathcal{S}_{k}$ denote the subspaces of
$L^{2}_{a}(A_{r})$ generated by
$\{z^{m}\in{L^{2}_{a}(A_{r})}:m=k(\texttt{mod }n)\}$ for
$0\leq{k}<n$. To study reducing subspaces for the multiplication
operator $M_{z^{n}}$ on $L^{2}_{a}(A_{r})$, we will use these $n$
reducing subspaces $\mathcal{S}_{k}(0\leq{k}<n)$ for $M_{z^{n}}$.
Note that
\[L^{2}_{a}(A_{r})=\mathcal{S}_{0}\oplus{\mathcal{S}_{1}}\oplus\cdot\cdot\cdot\oplus{\mathcal{S}_{n-1}},\]
and so for any $f\in{L^{2}_{a}(A_{r})}$, we have a unique
orthogonal decomposition
\begin{equation}\label{11}f=f_{0}+f_{1}+\cdot\cdot\cdot+f_{n-1},\end{equation}
where $f_{k}\in{\mathcal{S}_{k}}(0\leq{k}<n)$.

In this section, we will need the following well known fact
\cite{DA}. For completeness, we provide a proof.

\begin{lem}\label{17}
  If $M_{F}:\mathcal{S}_{k}\rightarrow{L^{2}_{a}(A_{r})}$ is a (bounded)
multiplication operator by a function $F$ on $A_r$, then $F
\in{H^{\infty}(A_{r})}$ and $\norm{F}_{\infty}\leq\norm{M_{F}}$.

\end{lem}
\begin{proof}First, since $F$ is the quotient of two analytic
functions ($F=(M_{F}z^{k})/z^{k}$), it is meromorphic on $A_r$.
For a fixed $z\in{A_{r}}$, let $\lambda_{z}$ denote the
point-evaluation functional on $L^{2}_{a}(A_{r})$ defined by
\[\lambda_{z}(f)=f(z)\]
for $f\in{L^{2}_{a}(A_{r})}$. Clearly, $\lambda_{z}$ is bounded,
and for $f_{k}\in{\mathcal{S}_{k}}$,
\[|F(z)\lambda_{z}(f_{k})|=|F(z)f_{k}(z)|=|\lambda_{z}(M_{F}(f_{k}))|\leq{\norm{\lambda_{z}}\norm{M_{F}}
\norm{f_{k}}}.\] It follows that
$|F(z)|\norm{\lambda_{z}}\leq\norm{\lambda_{z}}\norm{M_{F}}$ for
any $z\in{A_{r}}$. Therefore, \[|F(z)|\leq\norm{M_{F}}\] for any
$z\in{A_{r}}$, and $F$ is analytic.

\end{proof}
Another familiar result classifies bilateral shifts up to
unitarily equivalence.

\begin{prop}\label{33}\cite{C}
If $S$, $T$ are two bilateral weighted shifts with weight
sequences $\{v_{m}\}$, $\{w_{m}\}$, and if there exists an integer
$k$ such that
\[|v_{m}|=|w_{m+k}|\texttt{ for all }m,\]
then $S$ and $T$ are unitarily equivalent. Moreover, the converse
is true.
\end{prop}

\begin{lem}\label{39}
For $i$, $j$ such that $0\leq{i}\neq{j}<n$, if
$M_{i}=M_{z^{n}}|\mathcal{S}_{i}$ and
$M_{j}=M_{z^{n}}|\mathcal{S}_{j}$, then $M_i$ and $M_j$ are not
unitarily equivalent.
\end{lem}
\begin{proof}
Let $e_{k}^{i}=\frac{z^{kn+i}}{\norm{z^{kn+i}}}$ where $k$ is an
integer. Then, $\{e_{k}^{i}:k\in{\mathbb{Z}}\}$ is an orthonormal
basis for $\mathcal{S}_{i}$.

First, we calculate the weights of the operator $M_i$. Since
\[M_{i}(e_{k}^{i})=\frac{z^{(k+1)n+i}}{\norm{z^{kn+i}}}=\frac{\norm{z^{(k+1)n+i}}}{\norm{z^{kn+i}}}e_{k+1}^{i},\]
the weights of $M_{i}$ are
\begin{equation}\label{26}\lambda_{k}=\frac{\norm{z^{(k+1)n+i}}}{\norm{z^{kn+i}}}\end{equation} for
$k\in{\mathbb{Z}}$.

Similarly, the weights of $M_j$ are
\begin{equation}\label{27}\mu_{k}=\frac{\norm{z^{(k+1)n+j}}}{\norm{z^{kn+j}}}\end{equation} for
$k\in{\mathbb{Z}}$.

Since $\norm{z^{n}}^{2}=\frac{1}{n+1}-\frac{r^{2(n+1)}}{n+1}$, by
Proposition \ref{33} we conclude that $M_{i}$ and $M_{j}$ are not
unitarily equivalent.

\end{proof}

Recall that determining the reducing subspaces of $M_{z^{n}}$ is
equivalent to finding the projections in the commutant of
$M_{z^{n}}$ (\cite{D}). Thus, in the following Proposition, we
characterize every bounded linear operator $T$ on
$L^{2}_{a}(A_{r})$ commuting with $M_{z^{n}}$.

\begin{prop}\label{90}
A bounded linear operator $T$ on $L^{2}_{a}(A_{r})$ commutes with
$M_{z^{n}}$ if and only if there are functions $F_{i}(0\leq{i}<n)$
in $H^{\infty}(A_{r})$ such that
\[Tf=\sum_{i=0}^{n-1}{F_{i}f_{i}},\]
where $f_{i}(0\leq{i}<n)$ denotes the functions in equation
(\ref{11}).

\end{prop}
\begin{proof}
$(\Leftarrow)$ Let
$M_{F_{i}}:L^{2}_{a}(A_{r})\rightarrow{L^{2}_{a}(A_{r})}$ be the
multiplication operator defined by $M_{F_{i}}(g)=F_{i}g$ for
$g\in{L^{2}_{a}(A_{r})}$. Then,
\begin{equation}\label{12}\sup_{\norm{f}=1}\norm{Tf}\leq\sum_{i=0}^{n-1}\norm{F_{i}f_{i}}
\leq(\sum_{i=0}^{n-1}\norm{M_{F_{i}}})(\sup_{i=0,\cdot\cdot,n-1}\norm{f_{i}})\leq
(\sum_{i=0}^{n-1}\norm{M_{F_{i}}})\norm{f}.\end{equation} It
follows that $\norm{T}<\infty$. Clearly, $TM_{z^{n}}=M_{z^{n}}T$.
\vskip0.2cm

$(\Rightarrow)$ Assume that $T$ is a (bounded) operator on
$L^{2}_{a}(A_{r})$ such that $TM_{z^{n}}=M_{z^{n}}T$. Then,
$T^{*}$ commutes with $M_{\lambda^{n}-z^{n}}^{*}$ for any
$\lambda\in{A_{r}}$. Clearly, for $\lambda\in{A_{r}}$,
$\ker{M_{\lambda^{n}-z^{n}}^{*}}$ is generated by
\[\{k_{\lambda{\omega_{k}}}:\omega_{k}=\texttt{exp}(2\pi{i}k/n)(0\leq{k}<n)\},\]
where $k_{\lambda{\omega_{k}}}$ is the Bergman kernel function at
$\lambda{\omega_{k}}$.

Since $T^{*}k_{\lambda}\in{\ker{M_{\lambda^{n}-z^{n}}^{*}}}$, we
have
\begin{equation}\label{1}T^{*}k_{\lambda}=\sum_{k=0}^{n-1}\overline{a_{k}(\lambda)}k_{\lambda{\omega_{k}}},\end{equation}
for uniquely determined complex numbers
$\{a_{k}(\lambda)\}_{k=0}^{n-1}$.

If $f\in{L^{2}_{a}(A)}$ and $z\in{A_{r}}$, then, by equation
(\ref{1}),
\begin{equation}\label{14}Tf(z)=(Tf,k_{z})=(f,T^{*}k_{z})=\sum_{k=0}^{n-1}a_{k}(z)f(z\omega_{k}).\end{equation}
Since $\omega_{k}^{n}=1$ for any $0\leq{k}<n$,
\begin{equation}\label{15}f_{i}(z\omega_{k})=\omega_{k}^{i}f_{i}(z)\texttt{ }
(0\leq{i,k}<n),\end{equation} where $f_{i}$ $(0\leq{i}<n)$ is the
function defined in equation (\ref{11}).

Since $\sum_{k=0}^{n-1}a_{k}(z)f(z\omega_{k})=
\sum_{k=0}^{n-1}a_{k}(z)f_{0}(z\omega_{k})+\sum_{k=0}^{n-1}a_{k}(z)f_{1}(z\omega_{k})+\cdot\cdot\cdot+
\sum_{k=0}^{n-1}a_{k}(z)f_{n-1}(z\omega_{k})$, (\ref{14}) and
(\ref{15}) imply that
\begin{equation}\label{16}Tf(z)=\sum_{k=0}^{n-1}a_{k}(z)f_{0}(z)+\sum_{k=0}^{n-1}a_{k}(z)\omega_{k}f_{1}(z)
+\cdot\cdot\cdot+\sum_{k=0}^{n-1}a_{k}(z)\omega_{k}^{n-1}f_{n-1}(z).
\end{equation}

For $0\leq{k}<n$, a function $F_k$ on $A_{r}$ is defined by
\begin{equation}\label{20}F_{k}(z)=\sum_{i=0}^{n-1}a_{i}(z)\omega_{i}^{k}.\end{equation} Then,
equation (\ref{16}) implies that
\[Tf=\sum_{i=0}^{n-1}{F_{i}f_{i}}.\]

To finish this proof, we have to show that $F_{i}(0\leq{i}<n)$ is
in $H^{\infty}(A_{r})$. Since $F_{k}(z)=\frac{T(z^{k})}{z^{k}}$
for $0\leq{k}<n$, $F_{k}$ is analytic on $A_r$.

By Lemma \ref{17}, $\norm{F_{k}}_{\infty}\leq\norm{M_{F_{k}}}$;
that is, $\norm{F_{k}}_{\infty}<\infty$ for any $0\leq{k}<n$.

\end{proof}
An analogous result is known for Toeplitz operators on the open
unit disk \cite{C}.

Because $M_{z^{n}}$ and
$M_{z^{n}}|\mathcal{S}_{0}\oplus{M_{z^{n}}|\mathcal{S}_{1}}\oplus\cdot\cdot\oplus{M_{z^{n}}|\mathcal{S}_{n-1}}$
are unitarily equivalent, in the following Proposition, we
determine the projections in the commutant of
$M_{z^{n}}|\mathcal{S}_{0}\oplus{M_{z^{n}}|\mathcal{S}_{1}}\oplus\cdot\cdot\oplus{M_{z^{n}}|\mathcal{S}_{n-1}}$.

\begin{prop}\label{18}
 For $0\leq{k}<n$, let
$M_{k}=M_{z^{n}}|\mathcal{S}_{k}$.

If $B=(B_{ij})_{(n\times{n})}$ is a projection such that
\begin{equation}\label{21}
\begin{pmatrix}
M_{0}&0&\cdot\cdot\cdot&0&0\\0&M_{1}&0&\cdot\cdot&0\\\cdot&\cdot&\cdot&\cdot&\cdot\\0&0&
\cdot\cdot\cdot&0&M_{n-1}\end{pmatrix}B=B\begin{pmatrix}
M_{0}&0&\cdot\cdot\cdot&0&0\\0&M_{1}&0&\cdot\cdot&0\\\cdot&\cdot&\cdot&\cdot&\cdot\\0&0&
\cdot\cdot\cdot&0&M_{n-1}\end{pmatrix},\end{equation} then there
are holomorphic functions $\varphi_{ij}(0\leq{i,j}<n)$ in
$H^{\infty}(A_{r})$ such that
\[B_{ij}=M_{\varphi_{ij}}.\]

Moreover, $\varphi_{ii}$ is a real-valued constant function on
$A_r$ for $0\leq{i}<n$, and $\varphi_{ij}\equiv{0}$ for
$i\neq{j}$.
\end{prop}
\begin{proof}
Since the operator $B$ commutes with $M_{z^{n}}$, by Proposition
\ref{90} and equation (\ref{20}), we have
\[Bf=\sum_{i=0}^{n-1}\varphi_{i}f_{i},\]
where $\varphi_{i}(z)=\sum_{k=0}^{n-1}a_{k}(z)\omega_{k}^{i}$ and
hence $\varphi_{i}\in{H^{\infty}(A_{r})}$.

Let $a_{k}(z)=\sum_{i=0}^{n-1}a_{ki}(z)$ where
$a_{ki}\in{\mathcal{S}_{i}}(i=0,1,\cdot\cdot\cdot,n-1)$. Then,
\[B=\begin{pmatrix}
\sum_{k=0}^{n-1}a_{k0}&\sum_{k=0}^{n-1}a_{k(n-1)}\omega_{k}&\cdot\cdot\cdot&\sum_{k=0}^{n-1}a_{k1}\omega_{k}^{n-1}
\\\sum_{k=0}^{n-1}a_{k1}&\sum_{k=0}^{n-1}a_{k0}\omega_{k}&\cdot\cdot&\sum_{k=0}^{n-1}a_{k2}\omega_{k}^{n-1}
\\\cdot&\cdot&\cdot\cdot\cdot&\cdot
\\\cdot&\cdot&\cdot\cdot\cdot&\cdot
\\\sum_{k=0}^{n-1}a_{k(n-1)}&\sum_{k=0}^{n-1}a_{k(n-2)}\omega_{k}&\cdot\cdot\cdot&
\sum_{k=0}^{n-1}a_{k0}\omega_{k}^{n-1}\end{pmatrix}.\] It follows
that
\[B=(M_{\varphi_{ij}})_{i,j=0}^{n-1},\]
where
$M_{\varphi_{ij}}:\mathcal{S}_{j}\rightarrow{\mathcal{S}_{i}}$ is
the multiplication operator defined by
\[M_{\varphi_{ij}}(f_{j})=\varphi_{ij}f_{j}\]
for $f_{j}\in{\mathcal{S}_{j}}$. By Lemma \ref{17},
$\varphi_{ij}(0\leq{i.j}<n)$ is in $H^{\infty}(A_{r})$.

Since
$M_{\varphi_{ii}}:\mathcal{S}_{i}\rightarrow{\mathcal{S}_{i}}$ and
$B $ is a projection, $M_{\varphi_{ii}}^{*}=M_{\varphi_{ii}}$.
Thus, $\varphi_{ii}$ is a real-valued holomorphic function and
hence $\varphi_{ii}$ is a constant function.

We now prove that $\varphi_{ij}=0$ if $i\neq{j}$. Suppose that
there are $l$ and $k$ in $\{0,1,\cdot\cdot\cdot,n-1\}$ such that
$l\neq{k}$ and $\varphi_{lk}\neq{0}$. By equation (10),
\begin{equation}\label{22}
M_{l}M_{\varphi_{lk}}=M_{\varphi_{lk}}M_{k}\texttt{ and }
M_{k}M_{\varphi_{kl}}=M_{\varphi_{kl}}M_{l}.
\end{equation}
Thus, equation (\ref{22}) implies that
\begin{equation}\label{23}M_{l}M_{\varphi_{lk}}M_{\varphi_{kl}}=
M_{\varphi_{lk}}M_{k}M_{\varphi_{kl}}=M_{\varphi_{lk}}M_{\varphi_{kl}}M_{l}.\end{equation}

Since $M_{\varphi_{kl}}=M_{\varphi_{lk}}^{*}$,
$M_{\varphi_{lk}}M_{\varphi_{kl}}:\mathcal{S}_{l}\rightarrow{\mathcal{S}_{l}}$
is a self-adjoint operator commuting with $M_{l}$.
 Then, in the same way as for $\varphi_{ii}$, we conclude that
$M_{\varphi_{lk}}M_{\varphi_{kl}}=M_{\varphi_{lk}\varphi_{kl}}$ is
a constant multiple of the identity operator; that is,
\begin{equation}\label{24}M_{\varphi_{lk}}M_{\varphi_{kl}}=c_{l}I_{\mathcal{S}_{l}}\texttt{ for }0\leq{l}<n,\end{equation} where
$I_{\mathcal{S}_{l}}$ is the identity operator on $\mathcal{S}_l$.
Note that $c_{l}>0$, since $M_{\varphi_{lk}}M_{\varphi_{kl}}$ is
positive and $\varphi_{lk}\neq{0}$.

Equations (\ref{22}) and (\ref{24}) imply that $M_{k}$ and $M_{l}$
are unitarily equivalent which is a contradiction by Lemma
\ref{39}.

\end{proof}

Finally, it is time to determine the reducing subspaces of the
multiplication operators $M_{z^{n}}(n\geq{2})$.

\begin{thm}\label{91}
For a given $n\geq{2}$, the multiplication operator
$M_{z^{n}}:L^{2}_{a}(A_{r})\rightarrow{L^{2}_{a}(A_{r})}$ has
$2^n$ reducing subspaces with minimal reducing subspaces
$\mathcal{S}_{0},\cdot\cdot\cdot ,\mathcal{S}_{n-1}$.

\end{thm}

\begin{proof}
Since $M_{z^{n}}$ and
$M_{z^{n}}|\mathcal{S}_{0}\oplus{M_{z^{n}}|\mathcal{S}_{1}}\oplus\cdot\cdot\oplus{M_{z^{n}}|\mathcal{S}_{n-1}}$
are unitarily equivalent, it is enough to consider the reducing
subspaces of
$M_{z^{n}}|\mathcal{S}_{0}\oplus{M_{z^{n}}|\mathcal{S}_{1}}\oplus\cdot\cdot\oplus{M_{z^{n}}|\mathcal{S}_{n-1}}$.

By Proposition \ref{18}, if $B=(B_{ij})_{(n\times{n})}$ is a
projection satisfying equation (10), then
\[B=\begin{pmatrix}
c_{0}&0&\cdot\cdot\cdot&0
\\0&c_{1}&\cdot\cdot&0
\\\cdot&\cdot&\cdot\cdot\cdot&\cdot
\\\cdot&\cdot&\cdot\cdot\cdot&\cdot
\\0&0&\cdot\cdot\cdot&
c_{n-1}\end{pmatrix},\] where $c_{i}(0\leq{i}<n)$ are real
numbers. Since $B^{2}=B$, it follows that $c_{i}=0,1$ for
$0\leq{i}<n$.

Therefore, the reducing subspaces of
$M_{z^{n}}|\mathcal{S}_{0}\oplus{M_{z^{n}}|\mathcal{S}_{1}}\oplus\cdot\cdot\oplus{M_{z^{n}}|\mathcal{S}_{n-1}}$
are
\begin{equation}\label{25}c_{0}\mathcal{S}_{0}\oplus{c_{1}\mathcal{S}_{1}}
\oplus\cdot\cdot\cdot\oplus{c_{n-1}\mathcal{S}_{n-1}},\texttt{
with } c_{i}=0,1.\end{equation}
 Thus,
this theorem is proven.

\end{proof}

\section{Reducing Subspaces for $T_{z^{n}}$}\label{43}

J.A. Ball (\cite{B}) and E. Nordgren (\cite{N}) studied the
problem of determining reducing subspaces for an analytic Toeplitz
operator on the Hardy space $H^2(\mathbb{D})$ of the open unit
disk.

In this section, for $n\geq{2}$, we determine the reducing
subspaces for the analytic Toeplitz operator $T_{z^{n}}$ on the
Hardy space $H^{2}(A_{r})$ of the annulus $A_{r}$. Note that, for
$T_{z^{n}}$ on $H^{2}(\mathbb{D})$, the problem has an easy but
sufficient answer, since $T_{z^{n}}$ and
$T_{z}\otimes{I_{\C^{n}}}$ are unitarily equivalent.

Recall that the Hardy space $H^{2}(A_{r})$ is the closure of
$R(A_{r})$ in $L^{2}(m)$, where $m$ is linear Lebesgue measure on
$\partial{A_{r}}$.

Let $\mathbb{S}_{k}$ denote the subspaces of $H^{2}(A_{r})$
generated by $\{z^{m}\in{H^{2}(A_{r})}:m=k(\texttt{mod }n)\}$ for
$0\leq{k}<n$.
 In the same way as in Section
\ref{28}, we will use these $n$ reducing subspaces
$\mathbb{S}_{k}$ for the Toeplitz operator
$T_{z^{n}}:H^{2}(A_{r})\rightarrow{H^{2}}(A_{r})$ defined by
\[T_{z^{n}}(f)=z^{n}f,\]
for $n\geq{2}$. Note that
\[H^{2}(A_{r})=\mathbb{S}_{0}\oplus{\mathbb{S}_{1}}\oplus\cdot\cdot\cdot\oplus{\mathbb{S}_{n-1}},\]
and so for any $f\in{H^{2}(A_{r})}$, we have a unique orthogonal
decomposition
\begin{equation}\label{32}f=f_{0}+f_{1}+\cdot\cdot\cdot+f_{n-1},\end{equation}
where $f_{k}\in{\mathbb{S}_{k}}(0\leq{k}<n)$. \vskip.1in

\begin{prop}\label{8}
For $i$, $j$ such that $0\leq{i}\neq{j}<n$, if
$T_{i}=T_{z^{n}}|\mathbb{S}_{i}$ and
$T_{j}=T_{z^{n}}|\mathbb{S}_{j}$, then $T_i$ and $T_j$ are not
unitarily equivalent.

\end{prop}
\begin{proof}
Note that
\[\norm{z^{n}}^{2}_{{H^{2}({A_r}})}=\frac{1}{2\pi}\int_{0}^{2\pi}|e^{in\theta}|^{2}d\theta+
\frac{1}{2\pi}\int_{0}^{2\pi}{r}^{2n}|e^{in\theta}|^{2}d\theta=1+{r}^{2n}.\]

Then, in the same way as in Lemma \ref{39}, the result is proven.

\end{proof}

Determining the reducing subspaces of $T_{z^{n}}$ is equivalent to
finding projections in the commutant of $T_{z^{n}}$. Since
$T_{z^{n}}$ and
$T_{z^{n}}|\mathbb{S}_{0}\oplus{T_{z^{n}}|\mathbb{S}_{1}}\oplus\cdot\cdot\oplus{T_{z^{n}}|\mathbb{S}_{n-1}}$
are unitarily equivalent, we consider the commutant of
$T_{z^{n}}|\mathbb{S}_{0}\oplus{T_{z^{n}}|\mathbb{S}_{1}}\oplus\cdot\cdot\oplus{T_{z^{n}}|\mathbb{S}_{n-1}}$
in the following Proposition.

 Since we also have a kernel function in this case, a
description similar to that of the commutant of $M_{z^n}$ on the
Bergman space $L^{2}_{a}(A_{r})$ is obtained;


\begin{prop}\label{19}
A bounded linear operator $T$ on $H^{2}(A_{r})$ commutes with
$T_{z^{n}}$ if and only if there are functions $G_{i}(0\leq{i}<n)$
in $H^{\infty}(A_{r})$ such that
\[Tf=\sum_{i=0}^{n-1}{G_{i}f_{i}},\]
where $f_{i}$ denotes the functions in equation (\ref{32}).
\end{prop}

By Proposition \ref{8}, $T_{z^{n}}|\mathbb{S}_{i}$ and
$T_{z^{n}}|\mathbb{S}_{j}$ are not unitarily equivalent where
$0\leq{i}\neq{j}<n$. Thus, in the same way as in Proposition
\ref{18}, we characterize a projection which is in the commutant
of
$T_{z^{n}}|\mathbb{S}_{0}\oplus{T_{z^{n}}|\mathbb{S}_{1}}\oplus\cdot\cdot\oplus{T_{z^{n}}|\mathbb{S}_{n-1}}$;

\begin{prop}\label{31}
 For $0\leq{k}<n$, let
$T_{k}=T_{z^{n}}|\mathbb{S}_{k}$.

If $F=(F_{ij})_{(n\times{n})}$ is a projection such that
\begin{equation}\label{21}\begin{pmatrix}
T_{0}&0&\cdot\cdot\cdot&0&0\\0&T_{1}&0&\cdot\cdot&0\\\cdot&\cdot&\cdot&\cdot&\cdot\\0&0&
\cdot\cdot\cdot&0&T_{n-1}\end{pmatrix}F=F\begin{pmatrix}
T_{0}&0&\cdot\cdot\cdot&0&0\\0&T_{1}&0&\cdot\cdot&0\\\cdot&\cdot&\cdot&\cdot&\cdot\\0&0&
\cdot\cdot\cdot&0&T_{n-1}\end{pmatrix},\end{equation} then there
are holomorphic functions $\varphi_{ij}(0\leq{i,j}<n)$ in
$H^{\infty}(A_{r})$ such that
\[F_{ij}=T_{\varphi_{ij}}.\]

Moreover, $\varphi_{ii}$ is a real-valued constant function on
$A_{r}$ for $0\leq{i}<n$, and $ \varphi_{ij}\equiv{0}$ if
$i\neq{j}$.
\end{prop}

Since $T_{z^{n}}$ and
$T_{z^{n}}|\mathbb{S}_{0}\oplus{T_{z^{n}}|\mathbb{S}_{1}}\oplus\cdot\cdot\oplus{T_{z^{n}}|\mathbb{S}_{n-1}}$
are unitarily equivalent,
 we have the following result;
\begin{thm}\label{9}
 For a given $n\geq{2}$, the Toeplitz operator
$T_{z^{n}}:H^{2}(A_{r})\rightarrow{H^{2}(A_{r})}$ has $2^n$
reducing subspaces with minimal reducing subspaces
$\mathbb{S}_{0},\cdot\cdot\cdot ,\mathbb{S}_{n-1}$.

\end{thm}

\section{Reducing Subspaces for Bilateral Weighted Shifts}
Note that the multiplication operator $M_{z}$ on the Bergman space
$L^{2}_{a}(A_{r})$ and the Toeplitz operator $T_{z}$ on the Hardy
space $H^{2}(A_{r})$ are both bilateral weighted shifts. Moreover,
in Section \ref{28} and Section \ref{43}, we showed that the
lattice of reducing subspaces for the operators
$(M_{z})^{n}(=M_{z^{n}})$ on the Bergman space $L^{2}_{a}(A_{r})$
and $(T_{z})^{n}(=T_{z^{n}})$ on the Hardy space $H^{2}(A_{r})$,
both have $2^{n}$ elements for $n\geq{2}$. Thus, it is natural to
ask the following question.\vskip.1in

 \textbf{Question }: Let $H$ be a separable Hilbert space, and
 $S:H\rightarrow{H}$ be a \emph{bilateral weighted
shift}. Then, for a given $n\geq{2}$, does the operator $S^{n}$
 have a discrete lattice of $2^n$ reducing subspaces?
\vskip.05in

In \cite{S}, Stessin and Zhu answered this question for powers of
unilateral weighted shifts generalizing the earlier results for
$T_z$ on $H^2$ and $M_z$ on $L^{2}_{a}$. That some condition is
necessary is shown by considering the weighted shift with weights
$(\cdot\cdot\cdot, \frac{1}{2}, 2, \frac{1}{2},
2,\cdot\cdot\cdot)$. In this case $T^{2}=I$ and hence the lattice
of reducing subspaces consists of all subspaces.



In this section, we generalize their results finding hypotheses to
answer this question in the affirmative for certain bilateral
weighted shifts with spectrum $A_r$.

Let $\{\beta(m)\}$ be a two-sided sequence of positive numbers
such that
\begin{equation}\label{48}\sup_{m}\lambda_{m}=\sup_{m}\beta(m+1)/\beta(m)<\infty.\end{equation}
 We consider the space of
two-sided sequences $f=\{\hat{f}(m)\}$ such that
\[\norm{f}^{2}=\norm{f}_{\beta}^{2}=\sum|\hat{f}(m)|^{2}(\beta(m))^{2}<\infty.\]
We shall use the notation
\[f(z)=\sum\hat{f}(m)z^{m},\]
whether or not the series converges for any (complex) value of
$z$. We shall denote this space as $L^{2}(\beta)$ for the Laurent
series case.

Recall that these spaces are Hilbert spaces with the inner product
\begin{equation}\label{51}(f,g)=\sum\hat{f}(m)\overline{\hat{g}(m)}(\beta(m))^{2}.\end{equation}

Let $M_{z}:L^{2}(\beta)\rightarrow{L^{2}(\beta)}$ be the linear
transformation defined by
\begin{equation}\label{34}
(M_{z}f)(z)=\sum\hat{f}(m)z^{m+1}.
\end{equation}
By (\ref{48}), $M_z$ is bounded (\cite{C}). (Note that
$\{\lambda_{m}\}$ are the weights.) If $g_{k}(z)=z^{k}$, then
$\{g_{k}\}$ is an orthogonal basis for $L^{2}(\beta)$.

We let $\mathbf{S}_{k}$ denote the subspace of $L^{2}(\beta)$
generated by
\[\{g_{m}\in{L^{2}(\beta)}:m=k(\texttt{mod }n)\}\] for
$0\leq{k}<n$. To study the reducing subspaces for the operator
$M_{z^{n}}:L^{2}(\beta)\rightarrow{L^{2}(\beta)}$ defined by
\begin{equation}\label{2}(M_{z^{n}}f)(z)=\sum\hat{f}(m)z^{m+n}
(f\in{L^{2}(\beta)}),\end{equation} we will use the $n$ reducing
subspaces $\mathrm{\mathbf{S}}_{k}(0\leq{k}<n)$ for $M_{z^{n}}$.
Note that
\[L^{2}(\beta)=\mathbf{S}_{0}\oplus{\mathbf{S}_{1}}\oplus\cdot\cdot\cdot\oplus{\mathbf{S}_{n-1}},\]
and so for any $f\in{L^{2}(\beta)}$, we have a unique orthogonal
decomposition
\begin{equation}\label{70}f=f_{0}+f_{1}+\cdot\cdot\cdot+f_{n-1},\end{equation}
where $f_{k}\in{\mathbf{S}_{k}}(0\leq{k}<n)$.

Consider the multiplication of formal Laurent series, $fg=h$:
\begin{equation}\label{36}
(\sum\hat{f}(m)z^{m})(\sum\hat{g}(m)z^{m})=\sum\hat{h}(m)z^{m},
\end{equation}
where, for all $m$,
\begin{equation}\label{37}\hat{h}(m)=\sum_{k}\hat{f}(k)\hat{g}(m-k).\end{equation} In general, we
will assume that the product (\ref{36}) is defined only if all the
series (\ref{37}) are absolutely convergent. $L^{\infty}(\beta)$
denotes the set of formal Laurent series
$\phi(z)=\sum\hat{\phi}(m)z^{m}(-\infty<m<\infty)$ such that
$\phi{L}^{2}(\beta)\subset{L}^{2}(\beta)$.

If $\phi\in{L^{\infty}(\beta)}$, then the linear transformation of
multiplication by $\phi$ on $L^{2}(\beta)$ will be denoted by
$M_{\phi}$.
\begin{prop}\cite{C}\label{41}
If $A$ is a bounded operator on $L^{2}(\beta)$ that commutes with
$M_{z}$, then $A=M_{\phi}$ for some $\phi\in{L^{\infty}(\beta)}$.

\end{prop}

\begin{prop}\cite{N1}\label{42}
For $\phi\in{L^{\infty}(\beta)}$, $M_\phi$ is a bounded linear
transformation, and the matrix $(a_{mk})$ of $M_\phi$, with
respect to the orthogonal basis $\{g_{k}\}$, is given by
\begin{equation}\label{38}a_{mk}=\hat{\phi}(m-k).\end{equation}

\end{prop}

\begin{prop}\cite{C}\label{75}
The operator $M_z$ on $L^{2}(\beta)$ is unitarily equivalent to
the operator $\widetilde{M}_z$ on $L^{2}(\tilde{\beta})$ if and
only if there is an integer $k$ such that
\[\frac{\beta(n+k+1)}{\beta(n+k)}=\lambda_{n+k}=\tilde{\lambda}_{n}=\frac{\tilde{\beta}(n+1)}{\tilde{\beta}(n)}\] for all $n$. Equivalently,
$L^{2}(\tilde{\beta})=z^{k}L^{2}(\beta)$, and
\begin{equation}\label{53}\norm{f}_{1}={\norm{z^{k}f}}
\texttt{      } (f\in{L^{2}(\tilde{\beta})}),\end{equation} where
$\norm{f}_{1}$ denotes the norm of $f$ in $L^{2}(\tilde{\beta})$.

\end{prop}

\begin{lem}\label{3}

If $\beta_{0}(k)=\beta(nk)$, then $M_{0}=M_{z^{n}}|\mathbf{S}_{0}$
is unitarily equivalent to $M_z$ on $L^{2}(\beta_{0})$.
\end{lem}
\begin{proof}
Let $T:\mathbf{S}_{0}\rightarrow{L^{2}(\beta_{0})}$ be the linear
transformation defined by
\begin{equation}\label{49}
T(z^{nm})=z^{m},
\end{equation}
where $m\in\mathbb{Z}$.

If $f\in{\mathbf{S}_{0}}$, then $f(z)=\sum_{m}\hat{f}(nm)z^{nm}$
and, by equation (\ref{51}),
\begin{equation}\label{50}
\norm{f}_{\mathrm{S}_{0}}^{2}=\sum_{m}|\hat{f}(nm)|^{2}(\beta(nm))^{2}
=\sum_{m}|\hat{f}(nm)|^{2}(\beta_{0}(m))^{2}=\norm{Tf}_{L^{2}(\beta_{0})}^{2}.
\end{equation}
Therefore, $T$ is an isometry.
 Clearly, for a given $g=\sum_{m}\hat{g}(m)z^{m}\in{L^{2}(\beta_{0})}$, there is an element
 $f=\sum_{m}\hat{f}(nm)z^{nm}\in\mathbf{S}_{0}$ such that
 $T(f)=g$,
 where $\hat{g}(m)=\hat{f}(nm)$; that is, $T$ is onto. It follows that
$T$ is unitary.

Clearly,
\begin{equation}\label{61}M_{z}T=TM_{0}.\end{equation} Therefore,
$M_{0}=M_{z^{n}}|\mathbf{S}_{0}$ is unitarily equivalent to $M_z$
on $L^{2}(\beta_{0})$.

\end{proof}




We focus on the bilateral shift operator $M_{z}$ on $L^{2}(\beta)$
with monotonically incresing weights $\{\lambda_{n}\}$. If the
weights $\{\lambda_{n}\}$ of $M_{z}$ on $L^{2}(\beta)$ satisfy
\[|\lambda_{n}|\leq{cr^{n}}\texttt{ }for\texttt{ }some\texttt{ }c>0\texttt{ }and\texttt{ }\lim_{n\rightarrow\infty}\lambda_{n}=1,\] then $\sigma(M_{z})$
is the annulus $A_r$ \cite{C}.

We will call such operators $M_z$ a \emph{monotonic-}$A_r$
\emph{weighted shift.}

First, we obtain the analogue of Proposition \ref{90} and
Proposition \ref{18} in the case of monotonic-$A_r$ weighted
shifts.
\begin{prop}\label{40}
 For $0\leq{i}<n$, let $M_{i}=M_{z^{n}}|\mathbf{S}_{i}$, and
 assume that $M_z$ is a monotonic-$A_r$
weighted shift. If $P=(P_{ij})_{(n\times{n})}$ is a projection
such that
\begin{equation}\label{21}
\begin{pmatrix}
M_{0}&0&\cdot\cdot\cdot&0&0\\0&M_{1}&0&\cdot\cdot&0\\\cdot&\cdot&\cdot&\cdot&\cdot\\0&0&
\cdot\cdot\cdot&0&M_{n-1}\end{pmatrix}P=P\begin{pmatrix}
M_{0}&0&\cdot\cdot\cdot&0&0\\0&M_{1}&0&\cdot\cdot&0\\\cdot&\cdot&\cdot&\cdot&\cdot\\0&0&
\cdot\cdot\cdot&0&M_{n-1}\end{pmatrix},\end{equation} then there
are elements $\varphi_{ij}(0\leq{i},j<n)$ in $L^{\infty}(\beta)$
such that
\[P_{ij}=M_{\varphi_{ij}}.\]

Moreover, $\varphi_{ii}$ is a positive constant function for
$0\leq{i}<n$, and $\varphi_{ij}\equiv{0}$ for $i\neq{j}$.
\end{prop}
\begin{proof}


For a given $0\leq{i}<n$, define a sequence of positive numbers
$\{\beta_{i}(k)\}$ by $\beta_{i}(k)=\beta(nk+i)$ for
$k\in{\mathbb{Z}}$, and let
$T_{i}:\mathbf{S}_{i}\rightarrow{L^{2}(\beta_{i})}$ be the linear
transformation defined by
\begin{equation}\label{49}
T_{i}(z^{nm+i})=z^{m},
\end{equation}
where $m\in{\mathbb{Z}}$.

If $f\in{\mathbf{S}_{i}}$, then
$f(z)=\sum_{m}\hat{f}(nm+i)z^{nm+i}$ and, by equation (\ref{51}),
\[
\norm{T_{i}f}_{L^{2}(\beta_{i})}^{2}=\sum_{m}|\hat{f}(nm+i)|^{2}(\beta_{i}(m))^{2}
=\sum_{m}|\hat{f}(nm+i)|^{2}(\beta(nm+i))^{2}=\norm{f}_{S_{i}}.
\]
Thus,
$T_{i}$ is isometric, and for a given
$g=\sum_{m}\hat{g}(m)z^{m}\in{L^{2}(\beta_{i})}$, there is an
element
 $f=\sum_{m}\hat{f}(nm+i)z^{nm+i}\in\mathbf{S}_{i}$ such that
 $T_{i}(f)=g$,
 where $\hat{g}(m)=\hat{f}(nm+i)$; that is, $T_{i}$ is unitary.

Since $M_{z}T_{i}=T_{i}M_{i}$, we have
$M_{i}={T_{i}}^{-1}M_{z}T_{i}$. Hence, $M_{i}P_{ii}=P_{ii}M_{i}$
implies that
${T_{i}}^{-1}M_{z}T_{i}P_{ii}=P_{ii}{T_{i}}^{-1}M_{z}T_{i}$ and so
$M_{z}(T_{i}P_{ii}{T_{i}}^{-1})=({T_{i}}P_{ii}{T_{i}}^{-1})M_{z}$.
By Proposition \ref{41},
\begin{equation}\label{62} T_{i}P_{ii}{T_{i}}^{-1}=M_{\varphi_{ii}},\end{equation}
for some $\varphi_{ii}\in{L^{\infty}}(\beta_{i})$. Thus, $P_{ii}$
is unitarily equivalent to the linear transformation
${M_{\varphi_{ii}}}$ for some
$\varphi_{ii}\in{L^{\infty}(\beta_{i})}$.

By Proposition \ref{42}, since $M_{\varphi_{ii}}(0\leq{i}<n)$ is
self-adjoint, for any integers $m$ and
 $p$,
 \begin{equation}\label{67}\hat{\varphi}_{ii}(m-p)=\overline{\hat{\varphi}_{ii}(p-m)},\end{equation}
 and
\begin{equation}\label{68}(M_{\varphi_{ii}}({z^{p}}),z^{m})=(z^{p},M_{\varphi_{ii}}({z^{m}})).\end{equation}
Equations (\ref{67}) and (\ref{68}) imply that
\begin{equation}\label{777}\hat{\varphi}_{ii}(m-p)\beta(m)^{2}
=\overline{\hat{\varphi}_{ii}(p-m)}\beta(p)^{2}=\hat{\varphi}_{ii}(m-p)\beta(p)^{2}.\end{equation}
In equation (\ref{777}), if $m\neq{p}$, without loss of
generality, we assume that $m<{p}$. Then, by equation (\ref{777}),
\[\hat{\varphi}_{ii}(m-p)=\hat{\varphi}_{ii}(m-p)\frac{\beta(p)^{2}}{\beta(m)^{2}}=\hat{\varphi}_{ii}(m-p)
\frac{\beta(m+1)^{2}}{\beta(m)^{2}}\frac{\beta(m+2)^{2}}{\beta(m+1)^{2}}\cdot\cdot\cdot
\frac{\beta(p)^{2}}{\beta(p-1)^{2}}.\] Thus, since
$\lambda_{k}=\frac{\beta(k+1)}{\beta(k)}$ for any $k$,
\begin{equation}\label{101}
\hat{\varphi}_{ii}(m-p)=\hat{\varphi}_{ii}(m-p)\lambda_{m}^{2}\lambda_{m+1}^{2}\cdot\cdot\cdot\lambda_{p-1}^{2}.
\end{equation}
Since $M_z$ is a monotonic-$A_r$ weighted shift, by equation
(\ref{101}), we conclude that $\hat{\varphi}_{ii}(m-p)=0$ if
$p\neq{m}$. Clearly, $\hat{\varphi}_{ii}(0)$ is a real number by
equation (\ref{67}). Thus, $\varphi_{ii}$ is a real-valued
constant function: that is,
\begin{equation}\label{102}M_{\varphi_{ii}}=c_{i}I_{H}\end{equation} for some $c_{i}\in{\mathbb{R}}$. By
equations (\ref{62}) and (\ref{102}), $P_{ii}=M_{\varphi_{ii}}$.

Finally, if $P_{lk}\neq{0}$ for some $0\leq{l}\neq{k}<n$, in the
same way as the proof of Proposition \ref{18}, we have that
$M_{k}$ and $M_{l}$ are unitarily equivalent which is a
contradiction, since the weights are distinct, the weights for
$M_k$ and $M_l$ are completely different. Hence, $M_k$ and $M_l$
can't be unitalily equivalent for any $0\leq{k}\neq{l}<n$ by
Proposition \ref{75}. Therefore, $P_{ij}=0$ if $i\neq{j}$.

\end{proof}

In the next Theorem, we discuss the reducing subspaces of the
bilateral weighted shift operator $M_{z^{n}}$ on $L^{2}(\beta)$
for a monotonic-$A_r$ weighted shift $M_z$. Since $M_{z^{n}}$ and
$M_{z^{n}}|{\mathbf{S}}_{0}\oplus{M_{z^{n}}|{\mathbf{S}}_{1}}\oplus\cdot\cdot\oplus{M_{z^{n}}|{\mathbf{S}}_{n-1}}$
are unitarily equivalent, we have the following result.

\begin{thm}\label{65}
If $M_z$ is a monotonic-$A_r$ weighted shift, then the operator
$M_{z^{n}}:L^{2}(\beta)\rightarrow{L^{2}(\beta)}$ has $2^n$
reducing subspaces for $n\geq{2}$.
\end{thm}

Although we could state hypothesis for a version of Theorem
\ref{65} in terms of the weights as M. Stessin and K. Zhu
(\cite{S}) do for the case of unilateral weighted shifts, we state
one concrete result which generalizes Theorem \ref{91} and Theorem
\ref{9}.

An operator $T$ is said to be \emph{hyponormal} if
$[T^{*},T]=T^{*}T-TT^{*}\geq{0}$ and a \emph{strict hyponormal} if
$\ker[T^{*},T]=\{0\}$. One concrete application of Theorem
\ref{65} is the Corollary \ref{92}.

\begin{cor}\label{92}
If $M_{z}$ on $L^{2}(\beta)$ is a strict hyponormal operator such
that $\sigma(M_{z})=A_{r}$, then $M_{z^{n}}$ has $2^{n}$ reducing
subspaces for $n\geq{2}$.
\end{cor}
\begin{proof}
The operator $M_z$ is a strict hyponormal if and only if
$\lambda_{n}<\lambda_{n+1}$ for all $n$.
\end{proof}

Since all subnormal weighted shift operators which are not
isometric are strictly hyponormal, our earlier Theorems (Theorem
\ref{91} and Theorem \ref{9}) follow from Theorem \ref{65}.

\section{Kernel Function Point of View}
In this section, we also assume that the shift operator $M_{z}$ on
$L^{2}(\beta)$ is invertible. Then, $\sigma(M_{z})$ is the annulus
$A=\{z\in\C:[r(M_{z}^{-1})]^{-1}\leq|z|\leq{r(M_{z})}\}$, where
$r(M_{z})(r(M_{z}^{-1}))$ denotes the spectral radius of
$M_{z}$($M_{z}^{-1}$, respectively) \cite{C}.
 In this section, we focus on the shift operator $M_{z}$ on
$L^{2}(\beta)$ with monotonic weights $\{\lambda_{n}\}$. In this
section, we assume that the weights $\{\lambda_{n}\}$ of $M_{z}$
on $L^{2}(\beta)$ are monotonic satisfying
\[\lim_{n\rightarrow{-\infty}}\frac{\lambda_{n}}{r^n}=1
\texttt{ }and \texttt{
}\lim_{n\rightarrow{\infty}}\lambda_{n}=1.\]

By a \emph{Laurent polynomial} we mean a finite linear combination
of the vectors $\{g_{n}\}(-\infty<n<\infty)$. Recall that for a
complex number $\omega$, $\lambda_{\omega}$ denotes the functional
of evaluation at $\omega$, defined on Laurent polynomials by
$\lambda_{\omega}(p)=p(\omega)$.

\begin{defn}
$\omega$ is said to be a \emph{bounded point evaluation} on
$L^{2}(\beta)$ if the functional $\lambda_{\omega}$ extends to a
bounded linear functional on $L^{2}(\beta)$.
\end{defn}


In this section, the hypotheses on the weights imply that every
point $\omega$ in $A_{r}$ is a bounded point evaluation. Thus, we
have the \emph{reproducing kernel} $k_{\omega}$ for $L^{2}(\beta)$
associated with the point $\omega\in{A_{r}}$.

\begin{lem}\label{100}
  If $M_{F}:\mathbf{S}_{k}\rightarrow{L^{2}(\beta)}$ is a (bounded)
multiplication operator by a function $F$ on $A_r$, then $F
\in{H^{\infty}(A_{r})}$ and $\norm{F}_{\infty}\leq\norm{M_{F}}$.

\end{lem}
\begin{proof}
Since every point $\omega$ in $A_{r}$ is a bounded point
evaluation, it is proven in the same way as in Lemma \ref{17}.
\end{proof}
\begin{prop}
A bounded linear operator $T$ on $L^{2}(\beta)$ commutes with
$M_{z^{n}}$ if and only if there are functions
$\phi_{i}(0\leq{i}<n)$ in $H^{\infty}(A_{r})$ such that
\begin{equation}\label{79}Tf=\sum_{i=0}^{n-1}{\phi_{i}f_{i}},\end{equation} where
$f_{i}(0\leq{i}<n)$ denotes the functions in equation (\ref{70}).
\end{prop}

\begin{proof}
In the same way as in Proposition \ref{90}, we have analytic
functions $\phi_{i}(0\leq{i}<{n})$ on $A_{r}$ satisfying equation
(\ref{79}).

Since $\phi_{i}(z)=\frac{T(z^{i})}{z^{i}}$ for $0\leq{i}<{n}$, by
Lemma \ref{100}, $\phi_{i}\in{H^{\infty}(A_{r})}$.

\end{proof}

\begin{prop}
 For $0\leq{k}<n$, let
$M_{k}=M_{z^{n}}|\mathbf{S}_{k}$.

If $B=(B_{ij})_{(n\times{n})}$ is a projection such that
\begin{equation}\label{21}
\begin{pmatrix}
M_{0}&0&\cdot\cdot\cdot&0&0\\0&M_{1}&0&\cdot\cdot&0\\\cdot&\cdot&\cdot&\cdot&\cdot\\0&0&
\cdot\cdot\cdot&0&M_{n-1}\end{pmatrix}B=B\begin{pmatrix}
M_{0}&0&\cdot\cdot\cdot&0&0\\0&M_{1}&0&\cdot\cdot&0\\\cdot&\cdot&\cdot&\cdot&\cdot\\0&0&
\cdot\cdot\cdot&0&M_{n-1}\end{pmatrix},\end{equation} then there
are holomorphic functions $\varphi_{ij}(0\leq{i,j}<n)$ in
$H^{\infty}(A_{r})$ such that
\[B_{ij}=M_{\varphi_{ij}}.\]

Moreover, $\varphi_{ii}$ is a real-valued constant function on
$A_r$ for $0\leq{i}<n$, and $\varphi_{ij}\equiv{0}$ for
$i\neq{j}$.
\end{prop}

\begin{proof}
Since the the weight $\{\lambda_{n}\}$ of $M_{z}$ on
$L^{2}(\beta)$ is monotonic, by Proposition \ref{75}, $M_{k}$ and
$M_{l}$ are not unitarily equivalent for any $0\leq{k}\neq{l}<n$.
Thus, by the same way in Proposition \ref{18}, it is proven.

\end{proof}

\begin{thm}\label{66}
For $0\leq{i}<n$, let $M_{i}=M_{z^{n}}|\mathbf{S}_{i}$. Then the
bilateral weighted shift operator
${M}_{z^{n}}:L^{2}(\beta)\rightarrow{L^{2}(\beta)}$ has $2^n$
reducing subspaces for $n\geq{2}$.
\end{thm}
\begin{proof}
In the same way in Theorem \ref{91}, it is proven.
\end{proof}






\section{A Complex Geometric Point of Views}
The adjoint of a hyponormal weighted shift with spectrum equal to
the closure of $A_r$ for $0<{r}<1$ and essential spectrum equal to
$\partial{A_{r}}$ belongs to a very special class of operators,
$B_{1}(A_{r})$. Recall that, for a bounded domain $\Omega$ in $\C$
and a positive integer $n$, the $B_{n}(\Omega)$-class was
introduced by M.Cowen and the first author in \cite{CD} and
consists of those bounded operators on a Hilbert space $H$ that
satisfy;\vskip.5cm

(1) ran ${(T-\omega)}$ is closed for $\omega\in\Omega$,

(2) $\dim\ker{(T-\omega)}=n$ for $\omega\in\Omega$, and

(3) $\bigvee_{\omega\in\Omega}\ker(T-\omega)=H$.\vskip.2cm

The operators $M_{z}^{*}$ and $T_{z}^{*}$ as well as the adjoints
of the bilateral weighted shifts $M_{z^{*}}$ defined in the
previous sections with $\sigma(M_{z})=\overline{A_{r}}$ and
$\sigma_{e}=\partial{A_{r}}$ belong to $B_{1}(A_{r})$, while their
$n$th powers, $M_{z^{n}}^{*}$, $T_{z^{n}}^{*}$ and
$M_{z^{n}}^{*}$, belong to $B_{n}(A_{r^{n}})$.

All operators $T$ in $B_{1}(A_{r})$ have a kernel function, $k_z$,
and $\ker{(T^{n}-\omega)}$ is the span of
$\Gamma_{\omega}=\{k_{\lambda{\omega_{k}}}:\omega_{k}=\texttt{exp}(2\pi{i}k/n)(0\leq{k}<n)\}$,
where $\overline{\lambda}^{n}=\omega$. Thus, a holomorphic frame
for the hermitian holomorphic bundle $E_{T}$ canonically defined
by $T^{n}$ is given by the sums of the appropriate functions in
$\lambda$ analogous to the subspace decomposition into powers of
$z$, $z^k$, where $k\equiv{i}$( modulo $n$), and $0\leq{i}<n$. In
the general case, these sections don't correspond to reducing
subspaces since these sections being pairwise orthogonal can be
shown to be equivalent to $T$ being a weighted shift.

Operators in the commutant of $T^n$ correspond to anti-holomorphic
bundle maps, which have a matrix representation once a
anti-holomorphic frame is chosen for $E_T$. That is what was
accomplished as a first step in the earlier sections. Reducing
subspaces correspond to projection-valued anti-holomorphic bundle
maps and are determined by the value at a single point. Again,
that is the result proved in each of the three cases in which the
bundle $E_T$ is presented as the orthogonal direct sum of $n$
anti-holomorphic line bundles.

The question of whether there are other reducing subspaces is
equivalent to the issue of representing this bundle as a different
orthogonal direct sum. These bundles all have canonical Chern
connections and hence a corresponding curvature. The fact that the
operators obtained by restricting $T^n$ to one of these reducing
subspaces corresponds to the fact that the curvature has distinct
eigenvalues at some point in $A_r$. This is a straight calculation
in the case of the disk but much less so for the annulus.

If we take a general $T$ in $B_{1}(A_{r})$, it seems that the
lattice of reducing subspaces has $2^{k}$ elements for some
$0<k\leq{n}$. That is the case for Toeplitz operators on the Hardy
space $H^{2}_{\omega}(A_{r})$, where $\omega\in{A_{r}}$ and the
measure used to define $H^{2}_{\omega}(A_{r})$ is harmonic measure
on $A_r$ for the point $\omega$. It is not clear just how to
settle the general case, however, since calculating the curvature
is probably not feasible. (Note in this case $T_z$ is not a
bilateral weighted shift.) Thus, we need to develope other
techniques to settle this question.

A more general question concerns operators $T$ in $B_{n}(\Omega)$
for more general $\Omega$. For $T_{z}\otimes{I_{\C^{n}}}$ on
$H^{2}(D)\otimes{\C^{n}}$, the lattice of reducing subspaces is
continuous and infinite with no discrete part. Does this happen
for any other examples besides $T_{\varphi}\otimes{I_{\C^{n}}}$
where $\varphi$ is in $H^{\infty}(\Omega)$ ?



\bibliographystyle{amsplain}

\end{document}